\documentclass{article}

    \usepackage[preprint]{neurips_2024}
\usepackage{tikz} %
\usetikzlibrary{arrows,positioning,angles,quotes,fit}
\usetikzlibrary{bayesnet}

\usepackage[utf8]{inputenc} %
\usepackage[T1]{fontenc}    %
\usepackage{hyperref}       %
\usepackage{url}            %
\usepackage{booktabs}       %
\usepackage{amsfonts}       %
\usepackage{nicefrac}       %
\usepackage{microtype}      %
\usepackage{xcolor}         %
\usepackage{mathrsfs}

\usepackage{multirow}

\hypersetup{
    colorlinks = true,
    linkcolor = {blue},
    citecolor = {blue}
}

\usepackage{amsmath,amsfonts,bm,dsfont,mathtools}

\def\<#1,#2>{\left\langle #1,#2 \right\rangle}

\newcommand{\norm}[1]{\|#1\|}

\newcommand{\abs}[1]{|#1|}

\newcommand{\newterm}[1]{{\bf #1}}

\def\eqref#1{equation~\ref{#1}}

\newcommand{\1}[1]{\mathds{1}_{#1}}

\def\rvz{{\mathbf{z}}}

\def\vx{{\bm{x}}}

\DeclareMathAlphabet{\mathsfit}{\encodingdefault}{\sfdefault}{m}{sl}
\SetMathAlphabet{\mathsfit}{bold}{\encodingdefault}{\sfdefault}{bx}{n}

\newcommand{\prob}{\mathbb{P}}
\newcommand{\E}[2][]{\mathbb{E}_{#1}[#2]} %

\newcommand{\R}{\mathbb{R}}

\newcommand{\var}[2][]{{ \operatorname{Var}_{#1}\left(#2\right)}}

\usepackage{algorithm, algorithmic}
\usepackage{amsmath, amssymb, graphicx, url, amsthm}
\usepackage{thmtools,thm-restate}
\newtheorem{theorem}{Theorem}

\newtheorem{lemma}{Lemma}

\newtheorem{definition}{Definition}
\newtheorem{remark}{Remark}

\newtheorem{example}{Example}
\newtheorem{claim}{Claim}

\usepackage[capitalize,noabbrev]{cleveref}

\crefname{assumption}{Assumption}{Assumptions}
\crefname{proposition}{Proposition}{Propositions}
\crefname{example}{Example}{Examples}
\crefname{remark}{Remark}{Remarks}
\crefname{claim}{Claim}{Claims}
\crefname{conjecture}{Conjecture}{Conjectures}

\usepackage{todonotes}

\definecolor{lyr}{rgb}{1,0,0}

\title{Probability Tools for Sequential Random Projection}

\author{%
  Yingru Li\thanks{The author would like to thank Professor Zhi-Quan (Tom) Luo for advising this project.} \\
    The Chinese University of Hong Kong, Shenzhen, China \\
  \texttt{yingruli@link.cuhk.edu.cn} \\
}

\begin{document}

\maketitle

\begin{abstract}
  We introduce the first probabilistic framework tailored for sequential random projection, an approach rooted in the challenges of sequential decision-making under uncertainty. The analysis is complicated by the sequential dependence and high-dimensional nature of random variables, a byproduct of the adaptive mechanisms inherent in sequential decision processes. Our work features a novel construction of a stopped process, facilitating the analysis of a sequence of concentration events that are interconnected in a sequential manner. By employing the method of mixtures within a self-normalized process, derived from the stopped process, we achieve a desired non-asymptotic probability bound. This bound represents a non-trivial martingale extension of the Johnson-Lindenstrauss (JL) lemma, marking a pioneering contribution to the literature on random projection and sequential analysis.

\end{abstract}

\section{Introduction}
\label{sec:intro}
The evolution of random projection from a dimensionality reduction technique to a cornerstone of sequential decision-making processes marks a significant leap in computational mathematics and machine learning. Random projection traditionally employs a matrix \( \Pi = (\rvz_1, \ldots, \rvz_d) \in \mathbb{R}^{M \times d} \) to transform a high-dimensional vector \( \vx = (x_1, \ldots, x_d)^\top \in \mathbb{R}^d \) into a lower-dimensional space, preserving the Euclidean geometry within a bounded error as guaranteed by the Johnson-Lindenstrauss (JL) lemma. This preservation of geometric relationships, crucial for the efficacy of data compression and analysis techniques, is foundational to the lemma's broad applicability, such as computer science~\citep{indyk2001algorithmic,muthukrishnan2005data}, signal processing~\citep{candes2006near,baraniuk2008simple} and numerical linear algebra~\citep{woodruff2014sketching}.

\subsection{Sequential random projection}
Recent advancements, especially in reinforcement learning, underscore the pressing need for computational models that not only accommodate but thrive on the epistemic uncertainties of sequential decision-making \citep{li2022hyperdqn,li2023efficient,li2024hyperagent}. In these dynamic environments, the application of random projection must navigate the added complexity of decisions \((x_t)_{t \ge 1}\) that are influenced by a history of previous decisions and projection vectors (\(\rvz_0, x_1, \rvz_1, \ldots, x_{t-1}, \rvz_{t-1}\)), introducing a layer of sequential dependence absent in static models. More precisely, the sequential relationship among the random variables is
\begin{align}
\label{eq:sequential-dependence}
    x_t = f_t(x_1, \rvz_1, \ldots, x_{t-1}, \rvz_{t-1}), \quad t \ge 1,
\end{align}
where $f_t$ describes the relationship at time $t$, and $\rvz_t$ is sampled from a distribution over $\R^M$ with independent source of randomness.
This sequential relationships can be also described in \cref{fig:dependence-graph} using graphical model.

  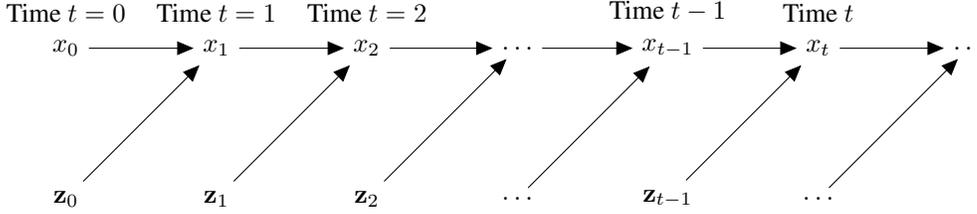
\begin{figure}[htbp]
    \begin{tikzpicture}[node distance=2cm, auto]
      \node (x0) {$x_0$};
      \node (Z0) [below of=x0] {$\rvz_0$};
      \node (x1) [right of=x0] {$x_1$};
      \node (z1) [below of=x1] {$\rvz_1$};
      \node (x2) [right of=x1] {$x_2$};
      \node (z2) [below of=x2] {$\rvz_2$};
      \node (xdot) [right of=x2] {$\ldots$};
      \node (zdot) [below of=xdot] {$\ldots$};
      \node (xtmins1) [right of=xdot] {$x_{t-1}$};
      \node (ztmins1) [below of=xtmins1] {$\rvz_{t-1}$};
      \node (xt) [right of=xtmins1] {$x_{t}$};
      \node (zdot2) [below of=xt] {$\ldots$};
      \node (xdot2) [right of=xt] {$\ldots$};

      \draw[->] (x0) -- (x1);
      \draw[->] (Z0) -- (x1);
      \draw[->] (x1) -- (x2);
      \draw[->] (z1) -- (x2);
      \draw[->] (ztmins1) -- (xt);
      \draw[->] (x2) -- (xdot);
      \draw[->] (z2) -- (xdot);
      \draw[->] (xdot) -- (xtmins1);
      \draw[->] (zdot) -- (xtmins1);
      \draw[->] (xtmins1) -- (xt);
      \draw[->] (xt) -- (xdot2);
      \draw[->] (zdot2) -- (xdot2);

      \node[above] at (x0.north) {Time $t=0$};
      \node[above] at (x1.north) {Time $t=1$};
      \node[above] at (x2.north) {Time $t=2$};
      \node[above] at (xtmins1.north) {Time $t-1$};
      \node[above] at (xt.north) {Time $t$};
    \end{tikzpicture}
    \caption{{{Sequential dependence}} of {{high-dimensional}} random variables due to the adaptive nature of sequential decision-making.}
    \label{fig:dependence-graph}
  \end{figure}

\paragraph{Unpacking the Challenges.}
This sequential dependence introduces significant analytical hurdles, diverging sharply from the assumptions that underpin classical random projection methods:
\begin{itemize}
    \item \emph{Traditional analysis of random projections} assumes the data vector $\vx = (x_1, \ldots, x_d)$ is fixed before the generation of the projection matrix $\Pi = (\rvz_1, \ldots, \rvz_d)$. All existing analysis for Johnson-Lindenstrauss or the studies for extreme singular values of random matrices~\citep{vershynin2012introduction} rely on some specific distributional properties of the random matrix $\Pi$. For example, conditions on independent entries across $\Pi$~\citep{indyk1998approximate,achlioptas2003database,matouvsek2008variants,vershynin2012introduction}, independent rows or columns across $\Pi$~\citep{vershynin2012introduction,kane2014sparser,cohen2018simple,li2024simple} are required to facilitate the concentration inequalities underlying the analysis.
    \item \emph{Analytical difficulties in sequential settings}: In the sequential setups, the decisions $x_t$ and projection vectors $\rvz_t$ are evolved together with sequential dependence described in \cref{eq:sequential-dependence}. Conditioned on the decision at time \(t\), the data \(x_t\), the preceding projection vectors (\(\rvz_0, \ldots, \rvz_{t-1}\)) lose their independence and identical distribution characteristics. This departure from independence directly challenges the foundational assumptions of analytical methods in random projection, complicating the task of maintaining accurate dimensionality reduction over sequential decisions. Specifically, without a clear understanding of the conditional distribution \(P_{(\rvz_{t'})_{t'<t} \mid x_t}\), traditional methods that rely on the specific distributional properties of projections cannot be straightforwardly applied. This limitation underscores a critical gap in our ability to predict and control the behavior of sequential random projections.
\end{itemize}

\paragraph{Innovations and Contributions.}

In addressing the challenges inherent in sequential random projection, this research inaugurates a comprehensive probabilistic framework, specifically designed to tackle the intricacies of sequential dependencies. This pioneering initiative represents a seminal contribution to the literature on random projection, particularly in sequential decision-making contexts. Our work is distinguished by two principal innovations:
\begin{itemize}
    \item[1.] \textbf{Technical innovations in stopped martingale}: Central to our contributions is the innovative construction of a stopped process, meticulously engineered to manage deviation behaviors within sequential processes. This construction crucially facilitates the precise control of concentration events over time, thereby enabling the analysis of the error bound in sequential setting.
    \item[2.] \textbf{Sequential extension of the Johnson–Lindenstrauss}: Through the employment of the method of mixtures, integrated with a self-normalized process, we have derived non-asymptotic probability bounds. These bounds represent a non-trivial extension of the Johnson–Lindenstrauss lemma into the realm of sequential analysis, equipping researchers and practitioners with a powerful analytical tool for the exploration of high-dimensional data in sequentially adpative processes.
\end{itemize}

The contributions of this research serve to bridge significant gaps in existing methodological frameworks, furnishing novel insights and methodologies for the application of random projection in sequential settings.
By laying a robust foundation for the analysis of dependencies among projection vectors in sequential contexts, our work not only surmounts the immediate challenges posed by sequential random projection but also forges a path for future investigations and applications in this critical intersection of mathematics and computational science. In doing so, it heralds a new paradigm in the analysis and application of random projection techniques for sequential, adaptive, high-dimensional data processing, promising to significantly influence subsequent research directions and applications in related fields.

\section{Probabilistic formalism \& statements}
\label{sec:tools}

\subsection{Probabilistic formalism}
One of the difficulties in the analysis is to deal with the sequential dependence structure among the random variables generated from the sequential-decision making problem such as bandit and reinforcement learning problems.
We define some important concept that would be useful in the analysis.
Let $(\Omega, \mathcal{F}, (\mathcal{F}_t)_{t \ge 0}, \prob)$ be a complete filtered probability space. We first consider the measurable properties within the filtered probability space.
\begin{definition}[Adapted process]
  For an index set $I$ of the form $\{t \in \mathbb{N}: t \ge t_0 \}$ for some $t_0 \in \mathbb{N}$, we say a stochastic process $(X_t)_{t \in I}$ is adapted to the filtration $(\mathcal{F}_t)_{t \in I}$ if each $X_t$ is $\mathcal{F}_t$-measurable.
\end{definition}
\begin{definition}[(Conditionally) $\sigma$-sub-Gaussian]
  \label{def:conditional-subgaussion}
  A random variable $X \in \R$ is $\sigma$-sub-Gaussian if
  \[
    \E{\exp( \lambda X ) } \le \exp\left( \frac{\lambda^2 \sigma^2}{2}  \right), \quad \quad \forall \lambda \in \R.
  \]
  Let $(X_t)_{t \ge 1} \subset \R$ be a stochastic process adapted to filtration $(\mathcal{F}_{t})_{t \ge 1}$. Let $\sigma = (\sigma_t)_{t\ge 0}$ be a stochastic process adapted to filtration $(\mathcal{F}_t)_{t\ge 0}$.
  We say the process is $(X_t)_{t \ge 1}$ is conditionally $\sigma$-sub-Gaussian if
  \[
    \E{\exp( \lambda X_t ) \mid \mathcal{F}_{t-1} } \le \exp \left( \frac{ \lambda^2 \sigma^2_{t-1}}{2}  \right), \quad a.s. \quad \forall \lambda \in \R.
  \]
  Specifically for the index $t+1$, we can say $X_{t+1}$ is ($\mathcal{F}_{t}$-conditionally) $\sigma_{t}$-sub-Gaussian.
  If $\sigma_t$ is a constant $\sigma$ for all $t \ge 0$, then we just say (conditionally) $\sigma$-sub-Gaussian.

  For a random vector $X \in \R^M$ or vector process $(X_t)_{t\ge 1} \subset \R^M$ in high-dimension, we say it is $\sigma$-sub-Gaussian is for every fixed $v \in \mathbb{S}^{M-1}$ if the random variable $\langle v, X \rangle$ , or the scalarized process $(\langle v, X_t \rangle)_{t \ge 1}$ is $\sigma$-sub-Gaussian.
\end{definition}

\begin{definition}[Almost sure unit-norm]
  \label{def:unit-norm}
  We say a random variable $X$ is almost sure unit-norm if $\norm{X}_2 = 1$ almost surely.
\end{definition}

Here, we give an example of distribution in $\R^M$ in \cref{ex:sphere} that fits the above mentioned definitions: the uniform distribution $\mathcal{U}(\mathbb{S}^{M-1})$ over unit sphere $\mathbb{S}^{M-1}$. The following recent result for the moment generating function (MGF) of Beta distribution is useful for the characterization of the sub-Gaussian property of the uniform distribution over the sphere.
\begin{lemma}[MGF of Beta distribution~\citep{li2024simple}]
  \label{lem:mgf-beta}
  For any $\alpha, \beta \in \R_+$ with $\alpha \ge \beta$,
  random variable $X \sim \operatorname{Beta}(\alpha, \beta)$ has variance $\var{X} = \frac{\alpha \beta}{(\alpha + \beta)^2 (\alpha + \beta + 1)}$ and the centered MGF $\E{\exp(\lambda(X - \E{X}))} \le \exp( \frac{\lambda^2 \var{X}}{2} )$.
\end{lemma}
For completeness, we provide the proof of \cref{lem:mgf-beta} in \cref{sec:additional-lemma}.
\begin{example}[Uniform distribution over sphere $\mathcal{U}(\mathbb{S}^{M-1})$]
  \label{ex:sphere}
  The random variable $\rvz \sim \mathcal{U}(\mathbb{S}^{M-1})$ is obviously unit-norm as by definition of the unit sphere $\mathbb{S}^{M-1}$.
  Also, according to \cref{lem:inner-dist}, the inner product between the random variable $\rvz$ and any unit vector $v \in \mathbb{S}^{M-1}$ follows a Beta distribution, i.e.,
  \[
    \langle \rvz, v \rangle \sim 2 \operatorname{Beta}\left(\frac{M-1}{2}, \frac{M-1}{2}\right)-1.
  \]
  Then, from \cref{lem:mgf-beta}, we could confirm $\rvz \sim \mathcal{U}(\mathbb{S}^{M-1})$ is $\sqrt{1/M}$-sub-Gaussian.
\end{example}
Additionally, we characterize the boundedness on the stochastic processes.
\begin{definition}[Square-bounded process]
  \label{def:x-boundedness}
  For an index set $I$ of the form $\{t \in \mathbb{N}: t \ge t_0 \}$ for some $t_0 \in \mathbb{N}$, the stochastic process $(X_t)_{t \in I}$ is $c$-square-bounded if $X_t^2 \le c$ almost surely for all $t \in I$.
\end{definition}

\subsection{Probability tools for sequential random projection}
In this section, we introduce the first probability tool for addressing sequential random projection, inspired by the challenges of sequential decision-making under uncertainty. The analysis is complicated by the sequential dependence and high-dimensionality of random variables, a consequence of the adaptive nature of sequential decision-making. Our approach leverages a novel and meticulously constructed stopped process that manages the behavior of a sequence of concentration events. By employing the method of mixtures as outlined by \citep{pena2009self} within a self-normalized process framework, we derive a non-asymptotic probability bound in \cref{thm:seq-jl}. This bound represents a non-trivial and significant martingale-based extension of the Johnson–Lindenstrauss (JL) lemma, {marking a novel contribution to the fields of random projection and sequential analysis alike}. Our technical innovation offers a probability statement for sequential random projection that is unparalleled in the literature, potentially sparking independent interest in both domains.

We use short notation for $[n] = \{1, 2, \ldots, n\}$ and $\mathcal{T} = \{ 0, 1, \ldots, T\} = \{ 0 \} \cup [T] $.
\begin{theorem}[Sequential random projection in adaptive process]
  \label{thm:seq-jl}
  Let $\varepsilon \in (0, 1)$ be fixed and  $(\mathcal{F}_{t})_{t\ge 0}$ be a filtration.
  Let $\rvz_0 \in \R^M$ be an $\mathcal{F}_0$-measurable random vector satisfies $\E{\norm{\rvz_0}^2} = 1$ and $\abs{\norm{\rvz}^2 - 1} \le (\varepsilon/2)$.
  Let $(\rvz_t)_{t \ge 1} \subset \R^M$ be a stochastic process adapted to filtration $(\mathcal{F}_{t})_{t \ge 1}$ such that it is $\sqrt{c_0/M}$-sub-Gaussian and each $\rvz_t$ is unit-norm.
  Let $(x_t)_{t\ge 1} \subset \R$ be a stochastic process adapted to filtration $(\mathcal{F}_{t-1})_{t \ge 1}$ such that it is $c_x$-square-bounded. Here, $c_0$ and $c_x$ are absolute constants.

  For any fixed $x_0 \in \R$, if the following condition is satisfied
    \begin{align}
    \label{eq:M-selection}
        M \ge \frac{16 c_0 (1+ \varepsilon)}{\varepsilon^2}  \left(\log \left( \frac{1}{\delta} \right) + \log \left(1 + \frac{c_x T}{x_0^2} \right) \right),
    \end{align}
  we have, with probability at least $1 - \delta$
  \begin{align}
  \label{eq:seq-rand-proj}
    \forall t \in \mathcal{T}, \quad
    (1 - \varepsilon) \left( \sum_{i=0}^t x_i^2 \right) \le \norm{\sum_{i=0}^t x_i \rvz_i}^2 \le (1 + \varepsilon) \left(\sum_{i=0}^t x_i^2 \right).
  \end{align}
\end{theorem}
\begin{remark}
  We say this is an ``sequential random projection'' argument because one can relate \cref{thm:seq-jl} to the traditional random projection setting where $\Pi_t = (\rvz_0, \ldots, \rvz_t) \in \R^{M \times t+1}$ is a random projection matrix and $\vx_t = (x_0, \ldots, x_t)^\top \in \R^{t+1}$ is the vector to be projected.
    The argument in \cref{eq:seq-rand-proj} translates to
    \begin{align}
    \label{eq:seq-rand-proj-2}
        \forall t \in \mathcal{T}, \quad (1-\varepsilon) \norm{\vx_t}^2 \le \norm{ \Pi_t \vx_t }^2 \le (1+ \varepsilon) \norm{\vx_t}^2.
    \end{align}
    When assuming independence between $\vx_t$ and $\Pi_t$ for all $t \in \mathcal{T}$, by simply applying union bound over time index $t \in \mathcal{T}$ with existing JL analysis, we can derive that the required dimension $M = O( \varepsilon^{-2} \log (T/\delta))$ is of the same order in \cref{eq:M-selection}.
    \textbf{However}, as discussed in \cref{sec:intro}, existing JL analytical techniques are \textit{not} able to handle the sequential dependence in our setup as $\vx_t$ is statistically dependent with $\Pi_t$ for $t \in \mathcal{T}$. 
  Therefore, \cref{thm:seq-jl} is also an innovation in the literature of both random projection and sequential analysis.
\end{remark}
\begin{remark}
  The unit-norm condition in the \cref{thm:seq-jl} can be removed. Then, more distribution of random vectors can be covered in our analytical framework.
  For example, the Gaussian random vector $\rvz \sim N(0, \frac{1}{M} I)$ is not unit-norm but we could rely on the centered moment generating function of $\norm{\rvz}^2$ by exploiting the properties of Chi-square distribution.
  We leave it for the future work.
\end{remark}
\begin{example}[Stylized stochastic process satisfying the condition in \cref{thm:seq-jl}.]
  \label{asmp:process-1}
  Let $(\rvz_{t})_{t \ge 0}$ are mutually independent random variables, each sampled from $\mathcal{U}(\mathbb{S}^{M-1})$. Let $x_0$ be fixed and $(x_t)_{t \ge 1}$ be the stochastic process with the following dependence structure, interleaved with the process $(\rvz_{t})_{t \ge 0}$:
  \begin{itemize}
    \item $x_t$ is dependent on $x_0, \rvz_0, x_1, \rvz_1, \ldots, x_{t-1}, \rvz_{t-1}$ as described in \cref{eq:sequential-dependence}.
    \item $\rvz_t$ is independent of $x_0, \rvz_0, x_1, \rvz_1, \ldots, x_{t-1}, \rvz_{t-1}, x_t $
  \end{itemize}
  This sequential dependence structure is also described in \cref{fig:dependence-graph}. 
  Define the filtration $(\mathcal{F}_t)_{t \ge 0}$ where $\mathcal{F}_{t} = \sigma(\rvz_0, x_1, \rvz_1, \ldots, x_t, \rvz_t, x_{t+1} )$.
  From \cref{ex:sphere}, we notice the process $(\rvz_t)_{t\ge 1}$ adapted to $(\mathcal{F}_t)_{t\ge 1}$  is $\sqrt{1/M}$-sub-Gaussian and unit-norm.
\end{example}

\section{Technical ideas \& details}
\label{sec:tech}
Before digging into the proof, we identify some important sequential structure and also clarity our proof idea in a intuitive level.
For each time $t \in \mathcal{T}$, let the short notation for the centered variable be
\[
  Y_t = \norm{ \sum_{i=0}^{t} x_i \rvz_i}^2 - \left( \sum_{i=0}^t x_i^2 \right)~~~~\text{and}~~~~S_t = \sum_{i=0}^t x_i^2.
\]
Our \emph{key observation} is that for any $t \in [T]$
\begin{align}
\label{eq:recursion}
  \norm{ \sum_{i=0}^{t} x_i \rvz_i}^2
   & = \norm{ \sum_{i=0}^{t-1} x_i \rvz_i + x_{t} \rvz_{t}}^2                                            \nonumber                                        \\
   & = \norm{ \sum_{i=0}^{t-1} x_i \rvz_i }^2 + 2 \left( \sum_{i=0}^{t-1} x_i \rvz_i \right)^\top  x_{t} \rvz_{t} + x_t^2 \norm{\rvz_t}^2
\end{align}
and thus we have the following relationship between $Y_{t}$ and $Y_{t-1}$ derived from \cref{eq:recursion},
\[
  Y_t - Y_{t-1} = 2 x_{t} \rvz_{t}^\top (\sum_{i=0}^{t-1} x_i \rvz_i ) + x_t^2 \left(\norm{\rvz_t}^2 -1 \right).
\]
Since $\rvz_t$ is unit-norm, we can further simplify the exposition
\begin{align}
  \label{eq:martingale-difference}
  Y_t - Y_{t-1} = 2 x_{t} \rvz_{t}^\top ( \sum_{i=0}^{t-1} x_i \rvz_i ).
\end{align}
Another key observation is that the difference term in \cref{eq:martingale-difference} is a function of on the $( \sum_{i=0}^{t-1} x_i \rvz_i )$ that is $\mathcal{F}_{t-1}$-measurable.
This implies, the difference term $(Y_{t} - Y_{t-1})$ can be controlled according to information in the history-dependent term $( \sum_{i=0}^{t-1} x_i \rvz_i ) = Y_{t-1} + S_{t-1}$. Intuitively, once the concentration behavior is bad, i.e., $Y_{t-1}$ has large deviation, it is highly possible to exhibit large deviation for $Y_{t'}$ in the later time index $t' \ge t$.

\subsection{Stopped process and exponential supermartingale}

To mathematically formalize this intuition, we introduce a definition of good event and stopping time for analysis.
\begin{definition}[Good event]
  \label{def:good-event}
  For each time $t \in \mathcal{T}$, we introduce the good event $E_t$ under which the strongly concentration behavior is guaranteed, suppose $\varepsilon \in (0, 1)$,
  \begin{align}
  \label{eq:good-event}
    E_t(\varepsilon) = \left\{ (1 - \varepsilon) \left( \sum_{i=0}^{t} x_i^2 \right) \le \norm{ \sum_{i=0}^t x_i \rvz_i}^2 \le (1 + \varepsilon) \left( \sum_{i=0}^{t} x_i^2 \right) \right\}.
  \end{align}
  With short notation,
  \begin{align*}
    E_t(\varepsilon) = \left\{ \abs{ Y_{t} } \le \varepsilon S_t \right\}.
  \end{align*}
\end{definition}
We also define the stopping time as the first time the bad event happens, i.e. the good event $E_t(\varepsilon)$ defined in \cref{eq:good-event} violates.
\begin{definition}[Stopping time]
  \label{def:stopping-time}
  For any fixed $\varepsilon$, we define the stopping time
  \begin{align}
  \label{eq:stopping-time}
      \tau(\varepsilon) = \min\{ t \in \mathcal{T}: \neg E_t(\varepsilon) \}.
  \end{align}
\end{definition}
Based on the stopping time, we construct a \newterm{stopped process}.
For $t \in [T]$, define the stopped difference term
\begin{align}
  \label{eq:def_X_t}
  X_{t}^{\tau} = (Y_t - Y_{t-1})  \1{t \le \tau}
\end{align}
such that the process $(X_t^{\tau})_{t\ge 1}$ is adapted to the filtration $(\mathcal{F}_{t})_{t\ge 1}$.

\begin{claim}
  \label{claim:supermartingale}
  Let $\tau$ be the stopping time $\tau(\varepsilon)$ defined in \cref{eq:stopping-time}.
  Let $(X^{\tau}_t)_{t\ge 1}$ be the stochastic process defined in \cref{eq:def_X_t} which is adapted to the filtration $(\mathcal{F}_{t})_{t\ge 1}$.
  Let $A^{\tau}_t=\sum_{i=1}^t X^{\tau}_i$.
  Further denote $(B^\tau_t)^2= \sum_{i=1}^t (C_i^\tau)^2 $ with
  \[
    (C_t^{\tau})^2 := \frac{4c_0}{M} x_t^2 (1+ \varepsilon)S_{t-1} \1{t \le \tau}.
  \]
  If the $(\mathcal{F}_{t})_{t\ge 1}$-adapted process $(\rvz_t)_{t\ge 1}$ is $\sqrt{{c_0}/{M}}$-sub-Gaussian and each $\rvz_t$ is unit-norm, then for any fixed $\lambda \in \R$
  \begin{align*}
    \left\{ M^{\tau}_{t}(\lambda) = \exp \left( \lambda A^{\tau}_t - \frac{\lambda^2}{2} (B^\tau_t)^2 \right)
    , \mathcal{F}_t, t\ge 1 \right\}
  \end{align*}
  is a \textbf{supermartingale} with mean $\le 1$.
\end{claim}
\begin{proof}[Proof of \cref{claim:supermartingale}]
  Note $\1{t \le \tau} = 1 - \1{\tau \le t-1}$ is $\mathcal{F}_{t-1}$-measurable. Thus, the random vector $$(\sum_{i=0}^{t-1} x_i \rvz_i) \1{t \le \tau} x_t \quad \text{is $\mathcal{F}_{t-1}$-measurable} .$$
  By the condition that the process $(\rvz_t)_{t\ge 1}$ is $\sqrt{c_0/M}$-sub-Gaussian, we conclude
  from the definition of conditionally sub-Gaussian in \cref{def:conditional-subgaussion},
  \begin{align}
  \label{eq:exponential-md}
    \E{ \exp( \lambda X^{\tau}_t ) \mid \mathcal{F}_{t-1}}
     & = \E{ \exp( 2\lambda x_t \langle \rvz_t,\sum_{i=0}^{t-1} x_i \rvz_i \rangle \1{t \le \tau}  ) \mid \mathcal{F}_{t-1}} \nonumber \\
     & \le \exp\left( \frac{\lambda^2}{2} (4 c_0/ M) x_t^2 \norm{\sum_{i=0}^{t-1} x_i \rvz_i }^2 \1{t \le \tau}  \right)   \nonumber   \\
     & \le \exp\left( \frac{\lambda^2}{2} (4 c_0/ M) x_t^2 (1+ \varepsilon)S_{t-1} \1{t \le \tau}  \right)       \nonumber                    \\
     & = \exp\left( \frac{\lambda^2}{2} (C_t^{\tau})^2 \right)
  \end{align}
  where the last inequality is because of the stopping time argument.
  Thus, the claim holds as
  \begin{align*}
      \E{ M_{t}^\tau (\lambda) \mid \mathcal{F}_{t-1} } = M_{t-1}^\tau(\lambda) \E{ \exp(\lambda X^\tau_t - \frac{\lambda^2}{2} (C^\tau_t)^2 ) \mid \mathcal{F}_{t-1} } \le M_{t-1}^\tau(\lambda),
  \end{align*}
  where the inequality is due to \cref{eq:exponential-md}.
\end{proof}

The following \citet{pena2009self}-type self-normalized bound would be useful to prove our main theoretical contribution of sequential random projection in \cref{thm:seq-jl}.
\begin{theorem}[Any-time self-normalized concentration bound]

  \label{thm:anytime-bound}
  Let $(\mathcal{F}_t)_{t\ge 0}$ be a filtration and  $\{ (A_t, B_t), t \ge 1 \}$ be a sequence of pairs of random variables satisfying
  that for all $\lambda \in \R$
  \begin{align*}
    \left\{ \exp\left(\lambda A_t - \frac{\lambda^2}{2} B_t^2 \right), \mathcal{F}_t, t \ge 1 \right\} ~\text{is a supermartingale with mean $\le 1$.}
  \end{align*}
  Then, for any fixed positive sequence $\left(L_t\right)_{t\ge 1}$, with probability at least $1-\delta$
  \begin{align*}
    \forall t \geq 1, \quad \abs{A_t} \leq \sqrt{2\left(B^2_t+L_t\right) \log \left(\frac{1}{\delta} \frac{\left(B^2_t+L_t\right)^{1 / 2}}{L_t^{1 / 2}}\right)}
  \end{align*}
\end{theorem}
The proof of \cref{thm:anytime-bound} can be found in \cref{sec:proof-anytime-bound}.

We also need the following trigger lemma for the initial preparation of the proof in \cref{thm:seq-jl}.
\begin{lemma}[Trigger lemma]
  \label{lem:trigger}
  For any sequence of event $(\mathcal{E}_t, t \in \mathcal{T})$, define the stopping time $\tau$ as the first time $t$ the event $\mathcal{E}_t$ is violated, i.e.
  \[
    \tau = \min \{ t \in \mathcal{T} : \neg \mathcal{E}_t \}.
  \]
  Then, the following equality holds for all $t \in \mathcal{T}$,
  \begin{align*}
    \{ \tau \le t \} = \neg \mathcal{E}_{t \wedge \tau}.
  \end{align*}
\end{lemma}

\subsection{Proof of \cref{thm:seq-jl}}
Now we are ready to provide the details of the proof, completing the intuition and mathematical construction.
\begin{proof}[Proof of \cref{thm:seq-jl}]
  We apply \cref{lem:trigger} for $\mathcal{E}_t = E_t(\varepsilon)$ and it follows
  \begin{align}
    \label{eq:apply-trigger}
    \prob\left( \exists t \in \mathcal{T}, \neg E_t(\varepsilon) \right)
     & = \prob(\tau \le T) \nonumber \\
     & = \prob \left( \neg E_{T \wedge \tau}(\varepsilon) \right)  \nonumber                            \\
     & = \prob \left( \abs{Y_{T \wedge \tau }} \ge \varepsilon S_{T \wedge \tau} \right)       \nonumber                    \\
     & = \prob \left( \abs{Y_0 + \sum_{t=1}^{T} (Y_{t} - Y_{t-1}) \1{t \le \tau}} \ge \varepsilon S_{T \wedge \tau} \right)
  \end{align}
  By the construction of stopped process
  $Y_{T \wedge \tau} - Y_0 = \sum_{t=1}^T X^{\tau}_t = A^{\tau}_T$.
  Then, \emph{our goal}, from \cref{eq:apply-trigger}, becomes to upper bound the RHS of \cref{eq:final-bound},
  \begin{align}
    \label{eq:final-bound}
    \prob\left( \exists t \in \mathcal{T}, (\neg E_t) \right)
     & = \prob \left(\abs{Y_0 + A^{\tau}_T} \ge \varepsilon S_{T \wedge \tau} \right)
  \end{align}
  By \cref{claim:supermartingale}, the pair of processes $(A^\tau_t, B^\tau_t)_{t\ge 1}$ with
  $$A^\tau_t = \sum_{i=1}^t X^\tau_i =
    \sum_{i=1}^t (Y_t - Y_{t-1})\1{t \le \tau}$$
  and $$(B^\tau_t)^2 = \sum_{i=1}^t (4c_0/M) x_t^2 (1+\varepsilon) S_{t-1} \1{t \le \tau}$$
  satisfies the conditions in \cref{thm:anytime-bound}.
  Then applying the \cref{thm:anytime-bound} on the pair of processes $(A^\tau_t, B^\tau_t)_{t\ge 1}$ yields that,
  with probability at least $1- \delta$,
  \begin{align*}
    \forall t \ge 1, \abs{A^{\tau}_t} \le
    \sqrt{2\left((B^\tau_t)^2 + L_t\right) \log \left(\frac{1}{\delta} \frac{\left((B^\tau_t)^2 + L_t \right)^{1 / 2}}{L_t^{1 / 2}}\right)}
  \end{align*}
  Since by the condition in \cref{thm:seq-jl}, we have $\abs{Y_0}\le (\varepsilon/2) x_0^2$.
  Now we want to argue that for any fixed $\varepsilon \in (0, 1)$, with suitable choice of $L_T$ and $M$, we have with probability at least $1 - \delta$
  \begin{align}
    \label{eq:final-goal}
    \abs{Y_{0} + A^\tau_T} \le
    \underbrace{\sqrt{2\left((B^\tau_T)^2 +L_T\right) \log \left(\frac{1}{\delta} \frac{\left((B^\tau_T)^2 +L_T\right)^{1 / 2}}{L_T^{1 / 2}}\right)} + (\varepsilon/2) x_0^2}_{(I)}
    \le \varepsilon S_{T \wedge \tau}.
  \end{align}
  \begin{claim}
    \label{claim:config-L-M}
    The following configuration suffices for \cref{eq:final-goal}:
    \begin{align*}
      L_T \le \frac{2 c_0 (1+ \varepsilon) x_0^4}{M} \quad \text{and} \quad
      M \ge  (16 c_0 (1+ \varepsilon)/\varepsilon^2)  \left(\log \left( \frac{1}{\delta} \right) + \log \left(1 + \frac{c_x T}{x_0^2} \right) \right).
    \end{align*}
  \end{claim}
  \begin{proof}[Proof of \cref{claim:config-L-M}]
    Recall the definition
    \(
      S_t = \sum_{t=0}^t x_i^2.
    \)
    We first calculate the term $(B^\tau_T)^2$ by our construction,
    \begin{align}
    \label{eq:B-upper-bound-1}
      (B^\tau_T)^2 & \le \frac{4 c_0}{M} \sum_{t=1}^{T \wedge \tau} x_t^2 \left( (1+ \varepsilon)S_{t-1}\right)                                               \\
                   & = \frac{4 c_0 ( 1 + \varepsilon) }{M} \sum_{t=1}^{T \wedge \tau} x_t^2 \left(  S_{T \wedge \tau} - \underbrace{(S_{T \wedge \tau} - S_{t-1})}_{\ge 0} \right) \nonumber \\
                   & \le \frac{4 c_0 (1+ \varepsilon)}{M} (S_{T \wedge \tau} - x_0^2)S_{T \wedge \tau}.
                       \label{eq:B-upper-bound-2}
    \end{align}
    From \cref{eq:B-upper-bound-1}, the almost sure upper bound of $(B^\tau_T)^2$ assuming $x_t^2 \le c_x$ is
    \[
      (B^\tau_T)^2 \le \frac{4 c_0 (1+ \varepsilon) }{M} \sum_{t=1}^T c_x (x_0^2 + (t - 1) c_x ) \le \frac{4c_0  (1+ \varepsilon) }{M} (c_x x_0^2 T + c_x^2 T^2/2 )
    \]
    Since $(a+b)^2 \le (1 + \lambda)( a^2 + (1/\lambda) b^2)$ for all fixed $\lambda \ge 0$, the term $(I)$ from \cref{eq:final-goal} becomes
    \begin{align*}
      (I)^2
      \le
      (1 + \lambda) \left( 2\left({(B^{\tau}_T)}^2+L_T\right) \log \left(\frac{1}{\delta} \frac{\left((B^{\tau}_T)^2+L_T\right)^{1 / 2}}{L_T^{1 / 2}}\right) + \frac{\varepsilon^2 x_0^4}{4 \lambda} \right)
    \end{align*}

    Let $L_T \le {4c_0(1+\varepsilon) \ell}/{M}$ and $\ell$ to be determined.
    \begin{align*}
      (I)^2
       & \le (1 + \lambda) \left( 2\left(B^2_T+L_T\right) \log \left(\frac{1}{\delta} \frac{\left(B^2_T+L_T\right)^{1 / 2}}{L_T^{1 / 2}}\right) + \frac{\varepsilon^2 x_0^4}{4\lambda} \right)                                                                \\
       & \le (1 + \lambda) \left( \frac{8c_0(1+\varepsilon)}{M} \left( (S_{T \wedge \tau} - x_0^2)S_{T \wedge \tau} + \ell \right) \log \left(\frac{1}{\delta} \sqrt{\frac{\left(c_x x_0^2 T + c_x^2 T^2/2 + \ell \right)}{\ell} }\right) + \frac{\varepsilon^2 x_0^4}{4\lambda} \right)
    \end{align*}
    Let $M \ge (8c_0(1+\varepsilon) /m)  \log \left( \frac{1}{\delta} \sqrt{ \frac{c_x x_0^2 T + c_x^2 T^2/2 + \ell }{\ell} } \right) $ and $m$ to be determined, we can simplify
    \begin{align*}
      (I)^2 \le (1 + \lambda)\left( m( (S_{T\wedge \tau} - x_0^2)S_{T \wedge \tau} + \ell ) + \frac{\varepsilon^2 x_0^4}{4\lambda} \right)
    \end{align*}
    Let $\ell = x_0^4/2$,
    $m = \varepsilon^2/(1+\lambda)$ and $\lambda=1$, we have
    \begin{align*}
      (I)^2 \le \varepsilon^2 (( S_{T \wedge \tau} - x_0^2 )S_{T \wedge \tau} + x_0^4/2 + x_0^4/2 )
      \le \varepsilon^2 S_{T \wedge \tau}^2
    \end{align*}
    where the last inequality is due to $x_0^2 = S_0 \le S_{T \wedge \tau}$ and $x_0^4 \le x_0^2 S_{T \wedge \tau}$.
    The conclusion is that we could select
    \begin{align*}
      M & \ge (16 c_0 (1+ \varepsilon)/\varepsilon^2)  \log \left( \frac{1}{\delta} \sqrt{ \frac{2 c_x x_0^2 T + c_x^2 T^2 + x_0^4 }{ x_0^4} } \right)    \\
        & = (16 c_0 (1+ \varepsilon)/\varepsilon^2)  \log \left( \frac{1}{\delta} \sqrt{ \frac{ (c_x T + x_0^2)^2 }{ x_0^4} } \right)    \\
        & =  (16 c_0 (1+ \varepsilon)/\varepsilon^2)  \left(\log \left( \frac{1}{\delta} \right) + \log \left(1 + \frac{c_x T}{x_0^2} \right) \right)
    \end{align*}
    and the auxiliary variable
    \begin{align*}
      L_T \le \frac{2c_0 (1+ \varepsilon) x_0^4}{M}.
    \end{align*}
  \end{proof}
\end{proof}

\section{Conclusions}
\label{sec:conclu}
This research has successfully established the first probabilistic framework specifically conceived for the domain of sequential random projection, addressing the complexities and challenges introduced by sequential dependencies and the high-dimensional nature of variables within sequential decision-making processes. By innovating a stopped process construction and extending the method of mixtures to include a self-normalized process, we have derived non-asymptotic probability bounds that significantly extend the Johnson–Lindenstrauss lemma into the realm of sequential analysis. These methodological advancements not only provide a robust foundation for accurately controlling concentration events and analyzing error bounds in a sequential context but also represent a seminal contribution to the intersection of computational mathematics and machine learning.

Our contributions offer a comprehensive solution to the analytical hurdles posed by the sequential dependence inherent in dynamic environments, specifically the loss of independence and identical distribution characteristics among projection vectors conditioned on sequential decisions. By addressing these challenges with precise technical innovations and extending foundational analytical tools to sequential settings, our work paves the way for future research and practical applications of random projection in sequential decision-making settings. In doing so, it heralds a paradigm shift towards a more nuanced understanding and application of random projection techniques in adaptive, high-dimensional data processing, promising to significantly influence future research directions and applications across related disciplines.

\bibliography{ref}
\bibliographystyle{plainnat}

\appendix
\section{Proof of \cref{thm:anytime-bound}: Method of mixtures}
\label{sec:proof-anytime-bound}
Robbins-Siegmund method of mixtures \citep{robbins1970boundary} originally is developed to evaluate boundary crossing probabilities for Brownian motion.
The method was further developed in the general theory for self-normalized process \citep*{de2004self,pena2009self,lai09martingales}.
\begin{remark}[Essential idea of Laplace approximation]
  If we integrate the exponential of a function that has a pronounced maximum, then we can expect that the integral will be close to the exponential function of the maximum.
  In our case, let
  \[
    M_t (\lambda) = \exp\left( \lambda A_t - \frac{\lambda^2}{2} B_t^2 \right)
  \]
  Informally, with this principle of Laplace approximation, we would have
  \begin{align*}
    \max_{\lambda} M_t(\lambda) \approx \int_{\Omega} M_t(\lambda) d h(\lambda)
  \end{align*}
  where $h$ is some measure on $\Omega$.
\end{remark}
\emph{The main benefit of replacing the maximum $\max_{\lambda} M_t(\lambda)$ with an integral $\bar{M}_t := \int_{\Omega} M_t(\lambda) d h(\lambda)$ is that we can handle the expectation  $\E{\bar{M}_t}$ easier while we don't know the upper bound on $\E{\max_{\lambda } M_t(\lambda)}$.} This is formalized in the following lemma.
\begin{lemma}
  \label{lem:mixture-bounded-expect}
  Let $(h_t)$ be a sequence of probability measures on $\Omega$. If $(M_t(\lambda), \mathcal{F}_t, t \ge 1)$ is a supermartingale with $\E{M_1(\lambda)}\le 1$ for all $\lambda \in \Omega$, then for any $t\ge 1$, the integrated random variable $\bar{M}_t=\int_{\Omega} M_t(\lambda) d h_t(\lambda)$ has expectation $\E{\bar{M}_t} \le 1$.

  Further, let $\tau$ be a stopping time with respect to filtration $(\mathcal{F}_t)_{t\ge 0}$, i.e. $\{\tau \le t\} \in \mathcal{F}_t, \forall t \ge 0$. Then $M_{\tau}(\lambda)$ is almost surely well-defined with expectation $\E{M_{\tau}(\lambda)} \le 1$ as well as $\E{\bar{M}_\tau} \le 1$.
\end{lemma}
\begin{proof}
  Using Fubini's theorem and the fact that $M_t(\lambda)$ is a supermartingale with $\E{M_t(\lambda)} \le\E{M_1(\lambda)} =1 $, we have
  \begin{align*}
    \E{\bar{M}_t} = \int_{\Omega} \E{ M_t(\lambda)} d h_t(\lambda) \le 1.
  \end{align*}
  For the expectation of stopped version $M_{\tau}(\lambda)$ and $\bar{M}_{\tau}$, we apply (supermartingale) optional sampling theorem.
\end{proof}
Finally, we are comfortable to drive the proof of the self-normalized concentration bounds.
\begin{proof}[Proof of \cref{thm:anytime-bound}]
  Let $\Lambda = (\Lambda_t)$ be a sequence of independent Gaussian random variable with densities
  \[
    f_{\Lambda_t}(\lambda) = c(L_t) \exp ( - \frac{1}{2} L_t \lambda^2 )
  \]
  where $c(A) = \sqrt{A / 2\pi}$ is a normalizing constant. We explicitly calculate $\bar{M}_t$ for any $t \ge 1$,
  \begin{align*}
    \bar{M}_t
     & =\int_{\R} \exp \left(\lambda A_t-\frac{\lambda^2}{2} B_t^2\right) f_{\Lambda_t}(\lambda) d \lambda                                                                                          \\
     & =\int_{\R} \exp \left(-\frac{1}{2}\left(\lambda-\frac{A_t}{B_t^2}\right)^2 B_t^2+\frac{1}{2} \frac{A_t^2}{B_t^2}\right) f_{\Lambda_t}(\lambda) d \lambda                                     \\
     & =\exp \left(\frac{1}{2} \frac{A_t^2}{B_t^2}\right) \int_{\R} \exp \left(-\frac{1}{2}\left(\lambda-\frac{A_t}{B_t^2}\right)^2 B_t^2\right) f_{\Lambda_t}(\lambda) d \lambda                   \\
     & =c\left(L_t\right) \exp \left(\frac{1}{2} \frac{A_t^2}{B_t^2}\right) \int_{\R} \exp \left(-\frac{1}{2}\left(\left(\lambda-A_t / B_t^2\right)^2 B_t^2+\lambda^2 L_t\right)\right) d \lambda .
  \end{align*}
  Completing the square yields
  \[
    \left(\lambda-\frac{A_t}{B_t^2}\right)^2 B_t^2+\lambda^2 L_t =\left(\lambda-\frac{A_t}{L_t+B_t^2}\right)^2\left(L_t+B_t^2\right)+\frac{A_t^2}{B_t^2}-\frac{A_t^2}{L_t+B_t^2}.
  \]
  By the change of variables $\lambda^{\prime}=\lambda-A_t /(L_t+B_t^2)$ in the following $(i)$,
  \begin{align*}
    \bar{M}_t & = c\left(L_t\right) \exp \left(\frac{1}{2} \frac{A_t^2}{L_t+B_t^2}\right) \int_{\R} \exp \left(-\frac{1}{2}\left(\lambda-\frac{A_t}{L_t+B_t^2}\right)^2\left(L_t+B_t^2\right)\right) d \lambda \\
              & \stackrel{(i)}{=} {c}\left(L_t\right) \exp \left(\frac{1}{2} \frac{A_t^2}{L_t+B_t^2}\right) \int_{\R} \exp \left(-\frac{1}{2}\left(\lambda^2\left(L_t+B_t^2\right)\right)\right) d \lambda     \\
              & = \frac{c\left(L_t\right)}{c\left(L_t+B_t^2\right)} \exp \left(\frac{1}{2} \frac{A_t^2}{L_t+B_t^2}\right).
  \end{align*}
  A final application of Markov’s inequality yields
  \begin{align*}
     & \quad \mathbb{P}\left[\abs{A_{\tau}} \geq \sqrt{2\left(L_{\tau}+B_{\tau}^2\right) \log \left(\frac{1}{\delta} \frac{\left(L_{\tau}+B_{\tau}^2\right)^{1 / 2}}{L_{\tau}^{1 / 2}}\right)}\right] \\
     & = \mathbb{P}\left[\frac{c\left(L_{\tau}\right)}{c\left(L_{\tau}+B_{\tau}^2\right)} \exp \left(\frac{1}{2} \frac{A_{\tau}^2}{L_{\tau}+B_{\tau}^2}\right) \geq \frac{1}{\delta}\right]           \\
     & \le \delta \cdot \mathbb{E}\left[ \frac{c\left(L_{\tau}\right)}{c\left(L_{\tau}+B_{\tau}^2\right)} \exp \left(\frac{1}{2} \frac{A_{\tau}^2}{L_{\tau}+B_{\tau}^2}\right)\right]                 \\
     & \stackrel{\text { i) }}{\leq} \delta \cdot \mathbb{E}\left[\bar{M}_{\tau}\right] \stackrel{\text { (ii) }}{\leq} \delta,
  \end{align*}
  where (i) uses the inequality for $\bar{M}_{\tau}$ derived above, and (ii) follows from \cref{lem:mixture-bounded-expect}.

  To get the anytime result in \cref{thm:anytime-bound}, we define the stopping time
  \begin{align*}
    \tau = \min\left\{ t\ge 1 : \abs{A_t} \geq \sqrt{2\left(L_t+B_t^2\right) \log \left(\frac{1}{\delta} \frac{\left(L_t+B_t^2\right)^{1 / 2}}{L_t^{1 / 2}}\right)} \right\}
  \end{align*}
  With an application of extended version of \cref{lem:trigger},
  and applying the previous inequality yields
  \begin{align*}
     & \quad \mathbb{P}\left[\exists t \ge 1, \abs{A_t} \geq \sqrt{2\left(L_t+B_t^2\right) \log \left(\frac{1}{\delta} \frac{\left(L_t+B_t^2\right)^{1 / 2}}{L_t^{1 / 2}}\right)}\right]             \\
     & = \mathbb{P}\left[\tau < \infty, \abs{A_\tau} \geq \sqrt{2\left(L_\tau+B_\tau^2\right) \log \left(\frac{1}{\delta} \frac{\left(L_\tau+B_\tau^2\right)^{1 / 2}}{L_\tau^{1 / 2}}\right)}\right] \\
     & \le \mathbb{P}\left[\abs{A_\tau} \geq \sqrt{2\left(L_\tau+B_\tau^2\right) \log \left(\frac{1}{\delta} \frac{\left(L_\tau+B_\tau^2\right)^{1 / 2}}{L_\tau^{1 / 2}}\right)}\right]              \\
     & \le \delta.
  \end{align*}
  This completes the proof.
\end{proof}

\section{Additional lemmas}
\label{sec:additional-lemma}
For the completeness, we provide the full details of \cref{lem:mgf-beta}, which is adapted from \cite{li2024simple}.
\begin{proof}
  We utilize the order-2 recurrence for central moments~\citep{skorski2023bernstein}:
  for a beta random varaiable $X \sim \operatorname{Beta}(\alpha, \beta)$, we have
  \begin{align*}
    \mathbb{E}\left[(X-\mathbb{E}[X])^p\right]= & \frac{(p-1)(\beta-\alpha)}{(\alpha+\beta)(\alpha+\beta+p-1)} \cdot \mathbb{E}\left[(X-\mathbb{E}[X])^{p-1}\right] \\ & +\frac{(p-1) \alpha \beta}{(\alpha+\beta)^2(\alpha+\beta+p-1)} \cdot \mathbb{E}\left[(X-\mathbb{E}[X])^{p-2}\right]
  \end{align*}
  Let $m_p := \frac{\mathbb{E}\left[(X-\mathbb{E}[X])^p\right]}{p!}$,
  when $\alpha \ge \beta$, it follows that $m_p$ is non-negative when $p$ is even, and negative otherwise.
  Thus, for even $p$,
  $$m_p \le \frac{1}{p} \cdot \frac{\alpha \beta }{(\alpha + \beta)^2(\alpha+\beta+p-1)} m_{p-2} \le \frac{\var{X}}{p} \cdot m_{p-2}.$$
  After repeating the above recursive equation for $p / 2$ times and combining with $m_p \leqslant 0$ for odd $p$, it yields the following relationships
  $$m_p \leqslant \begin{cases}\frac{\var{X}^{\frac{p}{2}}}{p ! !} & p \text { even } \\ 0 & p \text { odd }\end{cases}.$$
  With the application of $p ! !=2^{p / 2}(p / 2) !$ for even $p$, for $t \geqslant 0$ we obtain
  \[
    \E{\exp(\lambda [X - \E{X}])} \leqslant 1+\sum_{p=2}^{+\infty} m_p \lambda^p = 1 + \sum_{p=1}^{+\infty} (\lambda^2 \var{X}/2)^p/p! = \exp \left(\frac{\lambda^2 \var{X}}{2}\right)
  \]
\end{proof}

\begin{lemma}
  \label{lem:inner-dist}
  For any fixed unit vector $u \in \mathbb{S}^{n-1}$, for any random vector $v \sim \mathcal{U}(\mathbb{S}^{n-1})$, the inner product $u^{\top} v$ is distributed as $2 \operatorname{Beta}\left(\frac{n-1}{2}, \frac{n-1}{2}\right)-1$.
\end{lemma}
\begin{proof}
  By {rotational invariance} of $\operatorname{Uniform}(\mathbb{S}^{n-1})$,
  the distribution of $u^{\top} v$ should be identical $\forall u \in \mathbb{S}^{n-1}$.
  WLOG, let us look at $u = e_1 = (1, 0, \ldots, 0)$  that would project $v$ to the first coordinate, i.e.,
  the value of $v^{\top} e_1=v_1$.
    {Let $v_1$ be defined as $X$, which is a random variable.}
  Probability density that $X=x \in[-1,1]$ is proportional to the surface area occupied between $x$ and $x+d x$ occupied by the other coordinates.
  The surface area is a frustum of a cone with base as a $n-1$ dimension shell of radius as $\sqrt{1-x^2}$ and a height of $d x$ with slope of the cone as $1 / \sqrt{1-x^2}$.
  Hence, the probability density is
  \begin{align*}
    f_{X}(x)
     & \propto \frac{\sqrt{1-x^2}^{n-2}}{\sqrt{1-x^2}} \propto (1-x)^{\frac{n-3}{2}} (1+x)^{\frac{n-3}{2}}
  \end{align*}
  We examine the transformed random variable $Y = (X+1)/2$, i.e. $X = 2Y -1$. By the Change-of-Variable Technique,
  \begin{align*}
    f_{Y}(y) = 2 f_X(2y - 1) \propto (2-2y)^{\frac{n-3}{2}} (2y)^{\frac{n-3}{2}} \propto (1-y)^{\frac{n-3}{2}}y^{\frac{n-3}{2}}
  \end{align*}
  Thus, \[Y \sim \operatorname{Beta}\left( \frac{n-1}{2}, \frac{n-1}{2} \right). \]
  Then, by the fact $X = 2Y - 1$,
  \[
    X \sim 2 \operatorname{Beta}\left( \frac{n-1}{2}, \frac{n-1}{2} \right) - 1.
  \]
\end{proof}

\end{document}